\documentclass[11pt]{amsart}
\usepackage{amssymb,amscd,amsthm,verbatim}
\usepackage{amsbsy}
\usepackage{latexsym}
\usepackage[all,cmtip]{xy}
\usepackage[mathscr]{euscript}

\usepackage{etex}
\usepackage{graphicx}
\usepackage{pictex}
\usepackage{epstopdf}
\usepackage{color}

\usepackage{bbm}

\date{\today}

\newcommand{\Z}{{\mathbb Z}}
\newcommand{\R}{{\mathbb R}}

\newcommand{\C}{{\mathbb C}}
\newcommand{\D}{{\mathbb D}}

\newcommand{\N}{{\mathbb N}}

\newcommand{\CA}{{\mathcal{A}}}

\newcommand{\CC}{{\mathcal{C}}}

\newcommand{\CE}{{\mathcal{E}}}

\renewcommand{\Re}{{\mathrm{Re} \,}}
\renewcommand{\Im}{{\mathrm{Im} \,}}

\newcommand{\ac}{{\mathrm{ac}}}

\newcommand{\sing}{{\mathrm{s}}}

\newcommand{\Leb}{{\mathrm{Leb}}}

\newtheorem{theorem}{Theorem}[section]
\newtheorem{lemma}[theorem]{Lemma}

\newtheorem{coro}[theorem]{Corollary}

\theoremstyle{definition}

\newtheorem{definition}[theorem]{Definition}

\newtheorem{remark}[theorem]{Remark}

\theoremstyle{plain}

\allowdisplaybreaks \numberwithin{equation}{section}

\newcommand{\set}[1]{\left\{#1\right\}}

\begin{document}

\title{Subordinacy Theory for Extended CMV Matrices}

\author[S.\ Guo]{Shuzheng Guo}
\address{Ocean University of China, Qingdao 266100, Shandong, China and Rice University, Houston, TX~77005, USA}
\email{gszouc@gmail.com}
\thanks{S.G.\ was supported by CSC (No. 201906330008) and NSFC (No. 11571327)}

\author[D.\ Damanik]{David Damanik}
\address{Department of Mathematics, Rice University, Houston, TX~77005, USA}
\email{damanik@rice.edu}
\thanks{D.D.\ was supported in part by NSF grant DMS--1700131 and by an Alexander von Humboldt Foundation research award}
\author[D.\ Ong]{Darren C.\ Ong}
\address{Department of Mathematics, Xiamen University Malaysia, 43900 Sepang, Selangor Darul Ehsan, Malaysia}
\email{darrenong@xmu.edu.my}
\thanks{D.O.\ was supported in part by a grant from the Fundamental Research Grant Scheme from the Malaysian Ministry of Education (Grant No: FRGS/1/2018/STG06/XMU/02/1) and a Xiamen University Malaysia Research Fund (Grant Number: XMUMRF/2020-C5/IMAT/0011)}

\begin{abstract}
We develop subordinacy theory for extended CMV matrices. That is, we provide explicit supports for the singular and absolutely continuous parts of the canonical spectral measure associated with a given extended CMV matrix in terms of the presence or absence of subordinate solutions of the generalized eigenvalue equation. Some corollaries and applications of this result are described as well.
\end{abstract}

\maketitle

\section{Introduction}

Subordinacy theory was first developed in the setting of continuum half-line Schr\"odinger operators by Gilbert and Pearson \cite{GP87}. Its primary aim is to relate the spectral decomposition of the operator in question to the behavior of the solutions of the associated generalized eigenvalue equation. The following correspondence is obvious: a value $E$ of the spectral parameter is an eigenvalue of the operator $H$ in question if and only if the equation $H u = Eu$ admits a non-zero solution $u$ that belongs to the domain of the operator. Modulo a suitable regularity property, this means that $u$ satisfies the designated boundary condition at the origin and is square-integrable at $+ \infty$. Since the pure point part of any spectral measure of $H$ is supported by the set of eigenvalues, it follows that we can extract the pure point part of any spectral measure by restricting this measure to the set of $E$'s for which the solution that obeys the boundary condition at the origin is square-integrable. Similarly, we extract the continuous part by restriction to the set of $E$'s for which the solution that obeys the boundary condition at the origin is not square-integrable. Gilbert-Pearson's subordinacy theory provides a similar partition related to the split between the singular part and the absolutely continuous part of a spectral measure: the crucial question is now whether the solution that obeys the boundary condition at the origin is or is not subordinate.

A follow-up paper by Gilbert \cite{G89} developed subordinacy theory for continuum Schr\"odinger operators on the whole line, and the resulting theory is completely analogous if one replaces ``obeying the boundary condition at the origin'' by ``being square-integrable/subordinate at $-\infty$'' in the discussion above.

Subsequently, subordinacy theory was developed in other settings as well: for Jacobi matrices by Khan and Pearson \cite{KP92} and for CMV matrices by Simon \cite{S05}. Furthermore there were simplifications and extensions of subordinacy theory by Remling \cite{R97}, Jitomirskaya-Last \cite{JL99, JL00}, Damanik-Killip-Lenz \cite{DKL00}, and Killip-Kiselev-Last \cite{KKL03}.

The papers mentioned above establish subordinacy theory for half-line and whole-line Schr\"odinger operators, for half-line and whole-line Jacobi matrices, and for standard (i.e., half-line) CMV matrices. There is no subordinacy theory yet for extended (i.e., whole-line) CMV matrices, and it is the purpose of this paper to fill this gap in the literature.

Thus we are naturally motivated by the fact that subordinacy theory is a fundamental result to be established for any operator family for which such a theory exists. It is usually the most convenient way to perform a spectral analysis of a given operator, precisely because the behavior of generalized eigenfunctions is easier to study than other properties of the operator in question that are relevant to the identification of its spectral type. We expect our work to be useful in the study of spectral properties for many classes of extended CMV matrices.

The remainder of this paper is structured as follows. We describe the setting, the main result, and some consequences of it in Section~\ref{s.2}. Some known results that will be used in the proofs are presented in Section~\ref{s.3}. Section~\ref{s.4} develops the version of the Jitomirskaya-Last inequalities from \cite{S05} that we need to analyze the left half-line of a given extended CMV matrix. The main subordinacy result is then proved in Section~\ref{s.5} and its applications are discussed in Section~\ref{s.6}.

\section{Setting and Main Result}\label{s.2}

In this section we describe the setting in which we work and state the main result, a description of support of the parts of spectral measures of extended CMV matrices in terms of solutions, along with some corollaries. We refer the reader to \cite{S04, S05} for general background, and we follow largely the notation from these monographs.

Let $\mu$ be a non-trivial probability measure on the unit circle $\partial\mathbb{D}=\{z\in\mathbb{C}:|z|=1\}$,
which means the support of $\mu$ contains infinitely many points. By the non-triviality assumption, the functions 1,$z$, $z^2,\cdots$ are linearly independent in the Hilbert space $\mathcal{H} = L^2(\partial\mathbb{D}, d\mu)$, and hence one can form, by the Gram-Schmidt procedure, the \emph{monic orthogonal polynomials} $\Phi_n(z)$, whose Szeg\H{o} dual is defined by $\Phi_n^{*} = z^n\overline{\Phi_n({1}/{\overline{z}})}$. There are constants $\{\alpha_n\}_{n\in\N}$ in $\mathbb{D}=\{z\in\mathbb{C}:|z|<1\}$, called the \emph{Verblunsky coefficients}, so that
\begin{equation}\label{eq01}
\Phi_{n+1}(z) = z \Phi_n(z) - \overline{\alpha}_n \Phi_n^*(z), \qquad \textrm{ for } n\in \N,
\end{equation}
which is the so-called \emph{Szeg\H{o} recurrence}. Conversely, every sequence $\{\alpha_n\}_{n\in\N}$ in $\mathbb{D}$ arises as a sequence of recurrence coefficients corresponding to a Gram-Schmidt procedure on a nontrivial probability measure on $\partial\mathbb D$.

In fact, if we normalize the monic orthogonal polynomials $\Phi_n(z)$ by
$$\varphi(z, n)=\frac{\Phi_n(z)}{\|\Phi_n(z)\|_{\mu}},$$
where $\|\cdot\|_{\mu}$ is the norm of $\mathcal{H}$. It is easy to see that ~(\ref{eq01}) is equivalent to
\begin{equation*}
\rho_n(x)\varphi(z , n+1 ) = z \varphi (z , n) - \overline{\alpha}_n \varphi^*(z , n).
\end{equation*}

Define
\begin{equation}\label{e.salphazdef}
S (\alpha, z)= \frac{1}{\rho}
\left(
\begin{matrix}
z & -\overline{\alpha}\\
-\alpha z & 1
\end{matrix}
\right),
\end{equation}
where $\rho = (1 - |\alpha|^2)^{1/2}$.

The Szeg\H{o} recursion can be written in a matrix form as follows:
\begin{equation}{\label{eq25}}
\left(
\begin{matrix}
\varphi(z , n+1)\\
\varphi^{*}(z , n+1)
\end{matrix}
\right)
=
S(\alpha_n, z)
\left(
\begin{matrix}
\varphi(z , n)\\
\varphi^{*}(z , n)
\end{matrix}
\right).
\end{equation}
Alternatively, one can consider a different initial condition and derive the \emph{orthogonal polynomials of the second kind}, by setting $\psi(z , 0) = 1$ and then
\begin{equation*}
\left(
\begin{matrix}
\psi(z , n+1)\\
- \psi^{*}(z , n+1)
\end{matrix}
\right)
=
S(\alpha_n, z)
\left(
\begin{matrix}
\psi(z , n)\\
- \psi^{*}(z , n)
\end{matrix}
\right).
\end{equation*}

The orthogonal polynomials may or may not form a basis of $\mathcal{H}$. However, if we apply the Gram-Schmidt procedure to $1, z, z^{-1}, z^2, z^{-2}, \ldots$, we will obtain a basis -- called the \emph{CMV basis}. In this basis, multiplication by the independent variable $z$ in $\mathcal{H}$ has the matrix representation
\begin{equation}{\label{CMV}}
\mathcal{C}=\left(
\begin{matrix}
\overline{\alpha}_0&\overline{\alpha}_1\rho_{0}&\rho_1\rho_0&0&0&\cdots\\
\rho_0&-\overline{\alpha}_1\alpha_{0}&-\rho_1\alpha_0&0&0&\cdots\\
0&\overline{\alpha}_2\rho_{1}&-\overline{\alpha}_2\alpha_{1}&\overline{\alpha}_3\rho_2&\rho_3\rho_2&\cdots\\
0&\rho_2\rho_{1}&-\rho_2\alpha_{1}&-\overline{\alpha}_3\alpha_2&-\rho_3\alpha_2&\cdots\\
0&0&0&\overline{\alpha}_4\rho_3&-\overline{\alpha}_4\alpha_3&\cdots\\
\cdots&\cdots&\cdots&\cdots&\cdots&\cdots
\end{matrix}
\right),
\end{equation}
where $\alpha = \set{\alpha_{n}}_{n\in\N} \subset \D$ and $\rho_{n} = \sqrt{1 - |\alpha_{n}|^2}$, for $n\in \N$. A matrix of this form is called a \emph{CMV matrix}.

Furthermore, an \emph{extended CMV matrix} is a special five-diagonal doubly infinite matrix in the standard basis of $\ell^2(\mathbb{Z})$ according to \cite[Section 4.5]{S04} and \cite[Section 10.5]{S05}, written as
\begin{equation}{\label{eCMV}}
\mathcal{E}=\left(
\begin{matrix}
\cdots&\cdots&\cdots&\cdots&\cdots&\cdots&\cdots\\
\cdots&-\overline{\alpha}_0\alpha_{-1}&\overline{\alpha}_1\rho_{0}&\rho_1\rho_0&0&0&\cdots&\\
\cdots&-\rho_0\alpha_{-1}&-\overline{\alpha}_1\alpha_{0}&-\rho_1\alpha_0&0&0&\cdots&\\
\cdots&0&\overline{\alpha}_2\rho_{1}&-\overline{\alpha}_2\alpha_{1}&\overline{\alpha}_3\rho_2&\rho_3\rho_2&\cdots&\\
\cdots&0&\rho_2\rho_{1}&-\rho_2\alpha_{1}&-\overline{\alpha}_3\alpha_2&-\rho_3\alpha_2&\cdots&\\
\cdots&0&0&0&\overline{\alpha}_4\rho_3&-\overline{\alpha}_4\alpha_3&\cdots&\\
\cdots&\cdots&\cdots&\cdots&\cdots&\cdots&\cdots
\end{matrix}\right),
\end{equation}
where $\alpha = \set{\alpha_{n}}_{n\in\Z} \subset \D$ and $\rho_{n} = \sqrt{1 - |\alpha_{n}|^2}$, for $n\in \Z$. In some settings it is more natural to consider extended CMV matrices, rather than standard CMV matrices. This is the case, for example, when the Verblunsky coefficients are generated by an invertible ergodic dynamical system. This class of coefficients contains the important special cases of almost periodic and random coefficients and some important parts of the theory for ergodic coefficients, such as for example Kotani theory \cite{FDG, GT94, S05}, require the consideration of the two-sided case.

The main goal of this paper is to provide a general approach to the study of the spectral properties of a given extended CMV matrix $\CE$ via the properties of the solutions of the associated generalized eigenvalue equation. To this end, let us first discuss the canonical spectral measure and then the generalized eigenvalue equation

Given an extended CMV matrix $\mathcal{E}$, the canonical spectral measure $\Lambda$ is given by the sum of the spectral measures of $\CE$ relative to the vectors $\delta_0, \delta_1$. It is well known that $\{ \delta_0, \delta_1 \}$ forms a spectral basis for the operator $\mathcal{E}$ (see, e.g.\ \cite[Lemma~3]{MO14}) and hence for every $\psi \in \ell^2(\Z)$, the spectral measure corresponding to $\mathcal{E}$ and $\psi$ is absolutely continuous with respect to $\Lambda$.

Consider the Lebesgue decomposition of $\Lambda$ into its pure point, singular continuous, and absolutely continuous parts,
$$
\Lambda = \Lambda_\mathrm{pp} + \Lambda_\mathrm{sc} + \Lambda_\mathrm{ac}.
$$
That is, $\Lambda_\mathrm{pp}$ is supported by a countable set, $\Lambda_\mathrm{sc}$ gives no weight to countable sets but is supported by some set of zero Lebesgue measure, and $\Lambda_\mathrm{ac}$ gives no weight to sets of zero Lebesgue measure.
Here we refer to the standard arc length measure on $\partial \D$ as the Lebesgue measure on $\partial \D$.

We also consider the singular part of $\Lambda$,
$$
\Lambda_\mathrm{s} = \Lambda_\mathrm{pp} + \Lambda_\mathrm{sc}
$$
and the continuous part of $\Lambda$,
$$
\Lambda_\mathrm{c} = \Lambda_\mathrm{sc} + \Lambda_\mathrm{ac}.
$$

Consider the corresponding eigenvalue equation
\begin{equation}{\label{eq1}}
\mathcal{E} u = z u,
\end{equation}
with boundary conditions subjected by
\begin{equation}{\label{eq21}}
\left(
\begin{matrix}
\varphi_{\omega}(0) & \psi_{\omega}(0)\\
\varphi^{*}_{\omega}(0) & -\psi^{*}_{\omega}(0)
\end{matrix}
\right)
=
\left(
\begin{matrix}
\cos\omega + i \sin\omega & \cos \omega + i \sin\omega\\
\cos \omega - i \sin \omega & -\cos \omega + i \sin \omega
\end{matrix}
\right).
\end{equation}

Here is the fundamental definition of subordinacy, introduced by Gilbert and Pearson \cite{GP87} in the Schr\"odinger case, adapted to the CMV setting:

\begin{definition}{\label{d.subordinatesolution}}
(a) Define for a sequence $a_{0}, a_{1}, \ldots$ and $x \in (0, \infty)$,
\[
\| a \|^2_{x} = \sum_{j=0}^{[x]} |a_{j}|^2 + (x - [x])|a_{j+1}|^2,
\]
where $[x] = $ greatest integer less than or equal to $x$. An analogous expression defines $\| a \|^2_{x}$ for $a_{-1}, a_{-2}, \ldots$ and $x \in (-\infty,-1)$: define $\lceil x\rceil $ as the greatest integer greater than or equal to $x$. Then
	\[
\| a \|^2_{x} = \sum_{j=1}^{-\lceil x\rceil} |a_{-j}|^2 + (\lceil x\rceil - x )|a_{-j-1}|^2.
\]
(b) Let $z\in\partial \D$. A solution $u$ of \eqref{eq1} is called \emph{subordinate at $+\infty$} if it does not vanish identically and obeys
$$
\lim_{x\rightarrow +\infty} \frac{\|u\|_x}{\|p\|_x} = 0,
$$
for any linearly independent solution $p$ of \eqref{eq1}.

Similarly, a solution $u$ of \eqref{eq1} is called \emph{subordinate at $-\infty$} if it does not vanish identically and obeys
$$
\lim_{x\rightarrow -\infty} \frac{\|u\|_x}{\|p\|_x} = 0,
$$
for any linearly independent solution $p$ of \eqref{eq1}.
\end{definition}

We are now ready to state the main result of this paper:

\begin{theorem}{\label{t.main1rev}}
Let $\mathcal{E}$ be an extended CMV matrix in $\ell^2(\mathbb{Z})$ and denote by $\Lambda$ its canonical spectral measure. Then, the three parts of the canonical spectral measure have the following supports defined in terms of the behavior of the solutions of \eqref{eq1}:
\begin{itemize}

\item[{\rm (a)}] Let
$$
\mathcal{P} = \left\{ z \in \partial \D : \eqref{eq1} \textrm{ has a solution that is square-summable at both } \pm\infty \right\}.
$$
Then $\Lambda_\mathrm{pp}(\partial \D \setminus \mathcal{P}) = 0$ and $\Lambda_\mathrm{c}(\mathcal{P}) = 0$.

\item[{\rm (b)}] Let
$$
\mathcal{S} = \left\{ z \in \partial \D : \eqref{eq1} \textrm{ has a solution that is subordinate at both } \pm\infty \right\}.
$$
Then $\Lambda_\mathrm{s}(\partial \D \setminus \mathcal{S}) = 0$ and $\Lambda_\mathrm{ac}(\mathcal{S}) = 0$. In particular, we have that $\Lambda_\mathrm{sc}(\partial \D \setminus (\mathcal{S} \setminus \mathcal{P})) = 0$ and $(\Lambda_\mathrm{pp} + \Lambda_\mathrm{ac})(\mathcal{S} \setminus \mathcal{P}) = 0$.

\item[{\rm (c)}] Let
\begin{equation}\label{e.apmdef}
\mathcal{A}_{\pm} = \left\{ z \in \partial \D : \eqref{eq1} \textrm{ has no solution that is subordinate at }\pm \infty \right\}
\end{equation}
and
$$
\mathcal{A} = \mathcal{A}_{+} \cup \mathcal{A}_{-}.
$$
Then $\Lambda_\mathrm{ac}(\partial \D \setminus \mathcal{A}) = 0$ and $\Lambda_\mathrm{s}(\mathcal{A}) = 0$. Moreover, $\mathcal{A}$ is an essential support of $\Lambda_{\mathrm{ac}}$, that is, for any measurable set $\mathcal{A}'$ with $\Lambda_{\mathrm{ac}}(\partial\D \backslash \mathcal{A}') = 0$, we have $\mathrm{Leb}(\mathcal{A} \backslash \mathcal{A}') = 0$.
\end{itemize}
\end{theorem}

\begin{remark}\label{r.ppstatement}
Part (a) is well known and stated here for completeness. The statement follows quickly from the spectral theorem; see, for example, the proof of \cite[Theorem~7.27.(a)]{W80} for the derivation in the self-adjoint case -- the argument is analogous in the unitary case. Indeed, as discussed in the introduction, the philosophy behind subordinacy theory is to identify a type of solution behavior that discriminates between the absolutely and singular parts of spectral measures, just like square-summability discriminates  between the continuous and pure point parts of spectral measures.
\end{remark}

Typical applications of this result rely on sufficient conditions for the absence or presence of subordinate solutions. For example, the absence of subordinate solutions follows from the boundedness of the transfer matrices, which are defined as follows:
$$
A(n,z) = \begin{cases} S(\alpha_n,z) \times \cdots \times S(\alpha_0,z) & n \ge 0 \\
S(-\overline{\alpha}_{n-2},z) \times S(-\overline{\alpha}_{n-1},z) \times \cdots \times S(-\overline{\alpha}_{-2},z)   & n \le -1 \end{cases},
$$
where $S(\cdot,z)$ is given by \eqref{e.salphazdef}. We will give more details in Section 4.  Specifically, we have the following statement:

\begin{coro}\label{c.bddsol}
Let
$$
\mathcal{B}_{\pm} = \set{ z \in \partial \D : \sup_{n \in \Z_\pm} \|A(n,z)\| < \infty }.
$$
Then, $\mathcal{B}_{\pm} \subseteq \mathcal{A}_{\pm}$ with $\mathcal{A}_{\pm}$ as defined in \eqref{e.apmdef}. In particular, the restriction of $\Lambda$ to each of $\mathcal{B}_{\pm}$ is purely absolutely continuous.
\end{coro}

In many cases of interest, the Verblunsky coefficients are dynamically defined. As a result, the associated Szeg\H{o} recursion can be expressed in terms of $\mathrm{SU}(1,1)$-valued cocycles over the base dynamical system in question. The boundedness property that feeds into Corollary~\ref{c.bddsol} is then often established via a suitable reducibility result. Let us state another corollary in the dynamically defined setting that implements this connection.

\begin{coro}\label{c.ddvcbddc}
Suppose that $T : \Omega \to \Omega$ is invertible and $f : \Omega \to \D$. This gives rise to $\omega$-dependent Verblunsky coefficients
$$
\alpha_n(\omega) = f(T^n \omega), \quad \omega \in \Omega, \; n \in \Z,
$$
and $\omega$-dependent extended CMV matrices $\mathcal{E}(\omega) = \mathcal{E}(\{ \alpha_n(\omega)\})$. Moreover, for each $z \in \partial \D$, consider the map $A_z : \Omega \to \mathrm{SU}(1,1)$ given by
$$
A_z(\omega) = z^{-1/2} S(f(\omega),z),
$$
where $S(\cdot,z)$ is given by \eqref{e.salphazdef}.

Denote by $\mathcal{R}$ the set of $z \in \partial \D$ for which there exist $B_z : \Omega \to \mathrm{SU}(1,1)$ bounded and $A_z^{(0)} \in \mathrm{SU}(1,1)$ elliptic such that for every $\omega \in \Omega$, we have $A_z(\omega) = B_z(T \omega) A_z^{(0)} B_z(\omega)^{-1}$.

Then, for every $\omega \in \Omega$, the canonical spectral measure associated with $\mathcal{E}(\omega)$, $\Lambda(\omega)$, is purely absolutely continuous on $\mathcal{R}$.
\end{coro}

Recall that an $\mathrm{SU}(1,1)$ matrix is called elliptic if its trace belongs to the real interval $(-2,2)$ (in this context it is useful to remind the reader that $\mathrm{SU}(1,1)$ and $\mathrm{SL}(2,\R)$ are canonically conjugate; see \cite[equation (10.4.27)]{S05}). The assumptions of Corollary~\ref{c.ddvcbddc} can be verified in a variety of situations, in analogy to the extensive literature on reducibility for quasi-periodic $\mathrm{SL}(2,\R)$ cocycles of sufficient regularity; see, for example, \cite{ARAC, AK06, H09} and references therein. This connection is presently being worked out \cite{L20+, W20+}.

\bigskip

We conclude this section with two applications of the description of the singular part of an extended CMV matrix in terms of subordinate solutions. Both of these applications are known via different methods, but the approach via subordinacy theory provides an interesting additional angle. Since these results are not new, we will not state them as formal corollaries.

The first application is the following statement: any extended CMV matrix $\mathcal{E}$ has simple singular spectrum. In the setting of extended CMV matrices, this result was first proved by Simon \cite{S05b}. However, a statement of this kind had been obtained earlier for second-order differential operators, originally proved by Kac \cite{K62, K63} and then proved via subordinacy theory by Gilbert \cite{G98}. The present paper provides the basis for Gilbert's approach to this result in the setting of extended CMV matrices.

Another application concerns a version of the Ishii--Pastur theorem for ergodic extended CMV matrices, proved via subordinacy theory, an approach proposed in the setting of Schr\"odinger operators by Buschmann \cite{B97}. Again we will not state this as a formal corollary since the result is already known, see \cite[Theorem~B.2]{FDG}, and merely point out that the Ishii-Pastur theorem is the inclusion $\subseteq$ in the identity stated in \cite[Theorem~B.2]{FDG}, and that this inclusion can be proved along similar lines as in \cite{B97} using Theorem~\ref{t.main1rev} above.

\section{Preliminaries}\label{s.3}

\subsection{Carath\'{e}odory Functions}

A \emph{Carath\'{e}odory function} is a holomorphic map from $\D$ to the right half plane $\set{z : \Re z > 0}$. We also say a function is an \emph{anti-Carath\'{e}odory function} when its negative is a Carath\'{e}odory function.
If we modify $\alpha(n_0) = e^{i\theta_0}$, then ~(\ref{eCMV}) becomes the direct sum of matrices acting on $\ell^2([n_0+1,\infty)\bigcap \Z)$ and  $\ell^2([n_0,-\infty)\bigcap \Z)$ of the form ~(\ref{CMV}). We label the halves as $\CC_{+}^{(n_0+1)}$ and $\CC_{-}^{(n_0)}$, respectively. We consider the case when $n_0 = -1$. Concretely, ~(\ref{eCMV}) becomes the direct sum of matrices acting on  $\ell^2 (\Z_+)$ and $\ell^2 (\Z_-)$ of the form ~(\ref{CMV}), where we write $\mathbb N := [0,\infty) \bigcap \Z$ and $\Z_- := [-1,-\infty) \bigcap \Z$.
Notice that $\mathcal{C}^{(0)}_{+,11} = - e^{i\theta_0} \overline{\alpha}_0$ and $\mathcal{C}^{(0)}_{+,21} = - e^{i\theta_0} \rho_0$.
One can find the correspondence between a given CMV matrix and its Carath\'{e}odory function in~\cite[Section 1.3]{S04}. Specifically, a Carath\'{e}odory function is the CMV analog of the $m$-function in the theory of Jacobi matrices, and is connected to the spectral theory of the CMV matrices.

Denote the Carath\'{e}odory function corresponding to $\CC^{(0)}_{+}$ by $F_{+}(z,0) = \int_{\partial \D}\frac{\zeta + z}{\zeta - z} \, d\Lambda_{+}(\zeta,0)$ and $\CC^{(-1)}_{-}$ by $F_{-}(z,-1) = -\int_{\partial \D}\frac{\zeta + z}{\zeta - z} \, d\Lambda_{-}(\zeta,-1)$, where $\Lambda_{+}(\zeta,0)$ and $\Lambda_{-}(\zeta,-1)$ are the spectral measures of $\CC^{(0)}_{+}$ and $\CC^{(-1)}_{-}$, respectively.

The Carath\'{e}odory function for $\CE$ is given by the formula
\begin{equation*}
F(z) = \int \frac{e^{i\theta} + z}{e^{i\theta} - z} \, d \Lambda(\theta),
\end{equation*}
where as above $\Lambda$ is the sum of the spectral measures of $\CE$ relative to the vectors $\delta_0, \delta_1$.

\subsection{The Gesztesy-Zinchenko Description}

The Gesztesy-Zinchenko (GZ) matrix from \cite{GZ06} is a key tool to encode the behavior of solutions to~\eqref{eq1}. As we follow the conventions from \cite{S04,S05}, let us point out that there are some differences between the notations in \cite{GZ06} and ours, which are as follows: $\alpha_n = -\overline{\alpha}_{n-1}$, $U(\set{\alpha_{n}}) = \mathcal{E}(\set{-\overline{\alpha}_{n-1}})$, $U_{+,0}(\set{\alpha_{n}}) = \CC^{(0)\top}_{+}(\set{-\overline{\alpha}_{n-1}})$, $u_{\pm}(z,n)=v_{\pm}(z,n)$ and $v_{\pm}(z,n)=u_{\pm}(z,n)$, where the left hand sides are their notations and the right hand sides are our notations.

First, recall that any extended CMV matrix $\mathcal{E}$ can be factorized into direct sums of $2\times 2$ matrices of the form
\begin{equation*}
\Theta=
\left(
\begin{matrix}
\overline{\alpha} & \rho\\
\rho & -\alpha
\end{matrix}
\right).
\end{equation*}
Let
\begin{equation*}
\mathcal{L:}=\bigoplus_{j\in\mathbb{Z}}\Theta(\alpha_{2j}) \quad {\rm and}\quad
\mathcal{M}:=\bigoplus_{j\in\mathbb{Z}}\Theta(\alpha_{2j+1}),
\end{equation*}
then $\mathcal{E}=\mathcal{L}\mathcal{M}$.

Set
\begin{equation*}
P(\alpha,z):=\frac{1}{\rho}
\left(
\begin{matrix}
-\alpha & z^{-1}\\
z & -\overline{\alpha}
\end{matrix}
\right)   \textrm{ and }
Q(\alpha,z):=\frac{1}{\rho}
\left(
\begin{matrix}
-\overline{\alpha} & 1\\
1 & -\alpha
\end{matrix}
\right),
\textrm{ for } z\in \C\backslash \set{0}.
\end{equation*}
Now, if $u$ is a complex sequence such that $\mathcal{E}u=zu$ and $v=\mathcal{M}u$,
one can easily see that $\mathcal{E}^{\top}v=zv$ holds.
By \cite[Proposition 2.1]{DEFHV}, the following equation holds for $n \in \N$, which can be extended to $n\in\Z$,
\begin{equation}{\label{eq10}}
\left(
\begin{matrix}
u(n+1)\\
v(n+1)
\end{matrix}
\right)
=
T(n,z)
\left(
\begin{matrix}
u(n)\\
v(n)
\end{matrix}
\right),
\end{equation}
where
\begin{equation*}
T(n,z)=
\begin{cases}
P(\alpha_n,z), \qquad & n\text{ is even,}\\
Q(\alpha_n,z),& n \text { is odd.}
\end{cases}
\end{equation*}

\begin{definition}{\label{def1}}
We denote by
$\left(
\begin{matrix}
u_{+}(z,n,n_{0})\\
v_{+}(z,n,n_{0})
\end{matrix}
\right)_{n \geq n_0}
$
and
$\left(
\begin{matrix}
p_{+}(z,n,n_{0})\\
q_{+}(z,n,n_{0})
\end{matrix}
\right)_{n \geq n_0}
$,
for $z \in \C \backslash \set{0}$, two linearly independent solutions of ~(\ref{eq10}) for $n \geq 0$ with the following initial conditions:
\begin{align}\label{eq4}
\left(
\begin{matrix}
u_{+}(z,n_0,n_0)\\
v_{+}(z,n_0,n_0)
\end{matrix}
\right)
& =
\begin{cases}
\left(
\begin{matrix}
1\\
1
\end{matrix}
\right) \quad n_{0} \textrm{ is even},\\
\left(
\begin{matrix}
1\\
z
\end{matrix}
\right) \quad n_{0} \textrm{ is odd},
\end{cases}
\\
\nonumber \left(
\begin{matrix}
p_{+}(z,n_0,n_0)\\
q_{+}(z,n_0,n_0)
\end{matrix}
\right)
& =
\begin{cases}
\left(
\begin{matrix}
1\\
-1
\end{matrix}
\right) \quad n_{0} \textrm{ is even},\\
\left(
\begin{matrix}
-1\\
z
\end{matrix}
\right) \quad n_{0} \textrm{ is odd}.
\end{cases}
\end{align}
Similarly, we denote by
$\left(
\begin{matrix}
u_{-}(z,n,n_0)\\
v_{-}(z,n,n_0)
\end{matrix}
\right)_{n \leq n_0}
$
and
$\left(
\begin{matrix}
p_{-}(z,n,n_0)\\
q_{-}(z,n,n_0)
\end{matrix}
\right)_{n \leq n_0}
$,
for $z \in \C \backslash \set{0}$, two linearly independent solutions of \eqref{eq10} for $n \leq -1$ with the following initial conditions:
\begin{align}\label{eq9}
\left(
\begin{matrix}
u_{-}(z,n_0,n_0)\\
v_{-}(z,n_0,n_0)
\end{matrix}
\right)
& =
\begin{cases}
\left(
\begin{matrix}
1\\
-z
\end{matrix}
\right) \quad n_{0} \textrm{ is even},\\
\left(
\begin{matrix}
-1\\
1
\end{matrix}
\right) \quad n_{0} \textrm{ is odd},
\end{cases}
\\
\nonumber \left(
\begin{matrix}
p_{-}(z,n_0,n_0)\\
q_{-}(z,n_0,n_0)
\end{matrix}
\right)
& =
\begin{cases}
\left(
\begin{matrix}
1\\
z
\end{matrix}
\right) \quad n_{0} \textrm{ is even},\\
\left(
\begin{matrix}
1\\
1
\end{matrix}
\right) \quad n_{0} \textrm{ is odd}.
\end{cases}
\end{align}
\end{definition}

\begin{remark}{\label{re1}}
The above definition is from \cite[Definition 2.4]{GZ06}. Here our $u_{\pm}$, $v_{\pm}$, $p_{\pm}$ and $q_{\pm}$ are their $r_{\pm}$, $p_{\pm}$, $s_{\pm}$ and $q_{\pm}$ respectively. By (\ref{eq25}) and (\ref{eq10}), we have
\begin{equation}{\label{eq12}}
u_{+}(z , n) =
\begin{cases}
z^{\frac{-(n+1)}{2}} \varphi^{*}(z , n)\quad & n \textrm{ is }  odd,\\
z^{\frac{-n}{2}} \varphi(z , n) &n \textrm{ is }  even,
\end{cases}
\end{equation}
\begin{equation}{\label{eq13}}
v_{+}(z , n) =
\begin{cases}
z^{\frac{-(n-1)}{2}} \varphi(z , n) \quad &n \textrm{ is }  odd,\\
z^{\frac{-n}{2}} \varphi^{*}(z , n) & n \textrm{ is }  even,
\end{cases}
\end{equation}
\begin{equation}{\label{eq27}}
p_{+}(z , n) =
\begin{cases}
-z^{\frac{-(n+1)}{2}} \psi^{*}(z , n)\quad & n \textrm{ is }  odd,\\
z^{\frac{-n}{2}} \psi(z , n) &n \textrm{ is }  even,
\end{cases}
\end{equation}
\begin{equation}{\label{eq28}}
q_{+}(z , n) =
\begin{cases}
z^{\frac{-(n-1)}{2}} \psi(z , n) \quad &n \textrm{ is }  odd,\\
-z^{\frac{-n}{2}} \psi^{*}(z , n) & n \textrm{ is }  even,
\end{cases}
\end{equation}
for $z \in \partial \D$.
\end{remark}

For simplicity to check (\ref{eq12})--(\ref{eq28}), we rewrite the equation for $\set{\varphi(z,n)}_{n\in \N}$, $\set{\psi(z,n)}_{n\in \N}$, $\set{u_{+}(z,n)}_{n\in \N}$ and $\set{v_{+}(z,n)}_{n\in \N}$. Once we have (\ref{eq12}) and (\ref{eq13}), (\ref{eq27}) and (\ref{eq28}) hold immediately. Indeed, for $\set{\varphi(z,n)}_{n\in \N}$ and $\set{\psi(z,n)}_{n\in \N}$, we have
\begin{align*}
\rho_n \varphi(z, n+1) & = z \varphi(z, n) - \overline{\alpha}_{n} \varphi^{*}(z,n), \\
\rho_{n} \varphi^{*}(z, n+1) & = -\alpha_{n} z \varphi(z,n) + \varphi^{*}(z,n).
\end{align*}
For $\set{u_{+}(z,n)}_{n\in \N}$ and $\set{v_{+}(z,n)}_{n\in \N}$, when $n$ is even,
\begin{align*}
\rho_n u_{+}(z, n+1) & = -\alpha_n u_{+}(z, n) + z^{-1} v_{+}(z,n), \\
\rho_{n} v_{+}(z, n+1) & = z u_{+}(z,n) - \overline{\alpha}_{n}v(z,n);
\end{align*}
when $n$ is odd,
\begin{align*}
\rho_n u_{+}(z, n+1) & = -\overline{\alpha}_n u_{+}(z, n) + v_{+}(z,n), \\
\rho_{n} v_{+}(z, n+1) & = u_{+}(z, n) - \alpha_n v_{+}(z,n).
\end{align*}
It follows that \eqref{eq12} and \eqref{eq13} hold.

\begin{lemma}(\cite[Corollary 2.16]{GZ06}{\label{le2}})
There are solutions
$\left(
\begin{matrix}
s_{\pm}(z, \cdot)\\
t_{\pm}(z, \cdot)
\end{matrix}
\right)_{n\in \Z}$
of \eqref{eq10}, unique up to constant multiples, so that for $z\in \C \backslash (\partial \D \bigcup \set{0})$,
\begin{equation*}
\left(
\begin{matrix}
s_{+}(z, \cdot)\\
t_{+}(z, \cdot)
\end{matrix}
\right) \in \ell^2(\mathbb N)^2,
\end{equation*}
\begin{equation*}
\left(
\begin{matrix}
s_{-}(z, \cdot)\\
t_{-}(z, \cdot)
\end{matrix}
\right) \in \ell^2(\Z_-)^2.
\end{equation*}
\end{lemma}

\begin{lemma}(\cite[Theorem 5.3]{S041}){\label{le01}}
Let $z \in \D$. Then
\begin{equation*}
\sum_{n=0}^{\infty}\left|
\left(
\begin{matrix}
\psi(z, n)\\
-\psi^{*}(z,n)
\end{matrix}
\right)
+ \beta
\left(
\begin{matrix}
\varphi(z, n)\\
\varphi^{*}(z,n)
\end{matrix}
\right)
\right|^2
< \infty
\end{equation*}
if and only if
\begin{equation*}
\beta = F_{+}(z,0).
\end{equation*}
\end{lemma}

Let $\eta_{+} (z,n) = \psi(z,n) + F_{+}(z,0) \varphi(z,n)$ and $\eta^{\circledast}_{+} (z,n) = -\psi^{*}(z,n) + F_{+}(z,0) \varphi^{*}(z,n)$. Consider the equation
\begin{equation}{\label{eq19}}
\Xi_n = S_n(z) \Xi_0
\end{equation}
with the boundary condition $\Xi_{0} = \left(\begin{matrix} 1 + F_{+}(z,0)\\ -1 + F_{+}(z,0)\end{matrix}\right)$, where $S_n(z) = S(\alpha_{n-1}, z)\cdots S(\alpha_{0},z)$.
Then
$\left(
\begin{matrix}
\eta_{+} (z,n) \\
\eta^{\circledast}_{+} (z,n)
\end{matrix}\right)
$ is the unique $\ell^2$ solution of ~(\ref{eq19}).

The Green's function $G$ (or resolvent function $(\CE - z)^{-1}$) for $\CE$ is computed by using formal eigenvalues to $\CC_{\pm}$ and $\CC^{\top}_{\pm}$.

\begin{lemma}(\cite[Lemma 3.1]{GZ06}){\label{le1}}
For $z \in \C \backslash (\partial \D \bigcup \set{0})$, let $\Lambda_{-}(\zeta, 0)$ be the support of the spectrum of $\CC^{(0)}_{-}$ and let $M_{-}(z,0)$ be an anti-Carath\'{e}odory function in \cite[(2.139)]{GZ06}, that is related to $F_{-}(z,-1)$ by
\begin{equation}{\label{eq8}}
M_{-}(z,0) = \frac{\Re (1 - \overline{\alpha}_{-1}) + i \Im (1 + \overline{\alpha}_{-1}) F_{-}(z,-1)} {i \Im (1 - \overline{\alpha}_{-1}) + \Re (1 + \overline{\alpha}_{-1}) F_{-}(z,-1)}.
\end{equation}
Let $s_{\pm}$ be $\ell^2$ solutions to $(\CC_{\pm} - z) s = 0$, and let $t_{\pm}$ be $\ell^2$ solutions to $(\CC^{\top}_{\pm} - z) t = 0$, normalized by
\begin{align*}
&s_+(z, 0) = 1+F_+(z,0), \qquad s_-(z,0)=1+M_-(z,0),\\
&t_+(z, 0) = -1+F_+(z,0), \qquad t_-(z,0) = 1-M_-(z,0).
\end{align*}
These $s_\pm, t_\pm$ are equivalent to the ones in Lemma~\ref{le2}. We may extend these solutions to solutions of $(\CE - z) w = 0$ and $(\CE^{\top} - z) w = 0$.

Then the resolvent function $(\CE - z)^{-1}(x,y)$ can be expressed as
\begin{equation}{\label{eq2}}
\frac{-1} {2 z (F_{+}(z,0) - M_{-}(z,0))} =
\begin{cases}
t_{-}(z,m) s_{+}(z,n), & \textrm{ if } m < n \textrm{ or } m = n \textrm{ odd},\\
t_{+}(z,m) s_{-}(z,n), & \textrm{ if } m > n \textrm{ or } m = n \textrm{ even}.
\end{cases}
\end{equation}
\end{lemma}

From the table on page 181 of \cite{GZ06}, we are in the case of $k_0=0$ and obtain that
\begin{align*}
& t_{-}(z,1) = \frac{1}{\rho_0}(z + \overline{\alpha}_0) + \frac{1}{\rho_0} (z-\overline{\alpha}_0) M_{-}(z,0)\\
& s_{+}(z,1) = \frac{1}{\rho_0}(-\frac{1}{z} -\alpha_0) + \frac{1}{\rho_0}(\frac{1}{z} - \alpha_0) F_{+}(z,0).
\end{align*}
Notice that our $s_{\pm}$, $t_{\pm}$ are their $v_{\pm}$, $\tilde{u}_{\pm}$, respectively.

\subsection{Green and Carath\'{e}odory}

One can write $G_{00} + G_{11}$ as
\begin{align*}
G_{00} & + G_{11} \\
& = -\frac{(-1 + F_{+}(z,0)) (1 + M_{-}(z,0))} {2 z(F_{+}(z,0) - M_{-}(z,0))} \\
& \quad - \frac{[z + \overline{\alpha}_{0} + M_{-}(z,0) (z - \overline{\alpha}_{0})] [-1 - \alpha_0 z + F_{+}(z,0) (1 - \alpha_{0} z)]} {2 \rho_0^2 z^2 (F_{+}(z,0) - M_{-}(z,0))}\\
& = \frac{\rho_0^2 z (1 - F_{+}(z,0) + M_{-}(z,0) - F_{+}(z,0)M_{-}(z,0)) + (z + \overline{\alpha}_{0})(1+\alpha_0 z)}{2\rho_0^2z^2(F_{+}(z,0)-M_{-}(z,0))} \\
& \quad + \frac{(z-\overline{\alpha}_0)(1+\alpha_{0}z)M_-(z,0) - (z+\overline{\alpha}_{0})(1-\alpha_{0}z)F_{+}(z,0)}{2\rho_0^2z^2(F_{+}(z,0)-M_{-}(z,0))}\\
& \quad - \frac{(z-\overline{\alpha}_{0})(1-\alpha_0 z)F_{+}(z,0)M_{-}(z,0)}{2\rho_0^2z^2(F_{+}(z,0)-M_{-}(z,0))}\\
& = \frac{(\rho_{0}^2 z + z + \overline{\alpha}_0 + \alpha_0 z^2 + |\alpha_0|^2z) + (\rho_{0}^2 z + z - \overline{\alpha}_0 + \alpha_0 z^2 - |\alpha_0|^2 z)M_-(z,0)}{2\rho_0^2z^2(F_{+}(z,0)-M_{-}(z,0))} \\
& \quad + \frac{(-\rho_{0}^2 z + \alpha_0 z^2 - z + |\alpha_0|^2 z - \overline{\alpha}_0) F_{+}(z,0)}{2\rho_0^2z^2(F_{+}(z,0)-M_{-}(z,0))}\\
& \quad + \frac{(-\rho_0^2 z +\alpha_0 z^2 - z  -|\alpha_0|^2 z + \overline{\alpha}_0)F_{+}(z,0)M_{-}(z,0)}{2\rho_0^2z^2(F_{+}(z,0)-M_{-}(z,0))}\\
& = \frac{(2 z + \overline{\alpha}_0 + \alpha_0 z^2) + (2\rho_0^2 z - \overline{\alpha}_0 + \alpha_0 z^2)M_{-}(z,0)}{2\rho_0^2z^2(F_{+}(z,0)-M_{-}(z,0))} \\
& \quad + \frac{(\alpha_0 z^2 - \overline{\alpha}_0 - 2\rho_0^2 z )F_{+}(z,0) + (\alpha_0 z^2 + \overline{\alpha}_0 -2z) M_{-}(z,0)F_{+}(z,0)}{2\rho_0^2z^2(F_{+}(z,0)-M_{-}(z,0))}\\
& = \frac{(2z + \overline{\alpha}_0 + \alpha_0 z^2) + (\alpha_0 z^2 - \overline{\alpha}_0)(F_{+}(z,0) + M_{-}(z,0))}{2\rho_0^2z^2(F_{+}(z,0)-M_{-}(z,0))} \\
& \quad + \frac{2\rho_0^2z(M_{-}(z,0) - F_{+}(z,0)) + (\alpha_0 z^2 + \overline{\alpha}_0 - 2z) M_{-}(z,0) F_{+}(z,0)}{2\rho_0^2z^2(F_{+}(z,0)-M_{-}(z,0))}.
\end{align*}
It follows that
\begin{align}{\label{eq3}}
2 z & (G_{00} (z) + G_{11} (z)) \\ \notag
& = - 2 + \frac{(\overline{\alpha}_{0} + 2 z + \alpha_{0} z^2) + (\alpha_{0} z^2 - \overline{\alpha}_{0}) (M_{-}(z,0) + F_{+}(z,0))} { \rho_{0}^2 z (F_{+}(z,0) - M_{-}(z,0))} \\ \notag
& \quad + \frac{(\overline{\alpha}_{0} - 2 z + \alpha_{0} z^2) M_{-}(z,0) F_{+}(z,0)} { \rho_{0}^2 z (F_{+}(z,0) - M_{-}(z,0))}.
\end{align}

Finally we note the connection between $G_{00} + G_{11}$ and the Carath\'{e}odory function $F$ corresponding to $\CE$ and $d\Lambda$. We have by definition
\begin{equation*}
F(z) = \int \frac{e^{i \theta} + z}{e^{i\theta} - z} d\Lambda (\theta).
\end{equation*}
Define
\begin{equation*}
d\Lambda_r(\theta) = \Re F(r e^{i\theta}) \frac{d\theta}{2\pi}.
\end{equation*}
It is well known that $d\Lambda_{r}$ converges to $d\Lambda$ weakly as $r\uparrow 1$. Moreover,
\begin{align*}
F(z) &=  \int \frac{e^{i \theta} + z}{e^{i\theta} - z} d\Lambda (\theta) \\
     &=  1 + 2 z \int \frac{1} {{e^{i\theta} - z}} d\Lambda (\theta) \\
     &=  1 + 2 z (G_{00} (z) + G_{11} (z)).
\end{align*}

\section{Jitomirskaya-Last Inequalities}\label{s.4}

In this section, we obtain a suitable version of the Jitomirskaya-Last inequality for the left half line CMV matrix
\begin{equation*}
\mathcal{C}_{-}=\left(
\begin{matrix}
\cdots&\cdots&\cdots&\cdots&\cdots&\cdots&\\
\cdots&\overline{\alpha}_{-4}\rho_{-5}&-\overline{\alpha}_{-4}\alpha_{-5}&\overline{\alpha}_{-3}\rho_{-4}&\rho_{-3}\rho_{-4}&0&\\
\cdots&\rho_{-4}\rho_{-5}&-\rho_{-4}\alpha_{-5}&-\overline{\alpha}_{-3}\alpha_{-4}&-\rho_{-3}\alpha_{-4}&0&\\
\cdots&0&0&\overline{\alpha}_{-2}\rho_{-3}&-\overline{\alpha}_{-2}\alpha_{-3}&\overline{\alpha}_{-1}\rho_{-2}&\\
\cdots&0&0&\rho_{-2}\rho_{-3}&-\rho_{-2}\alpha_{-3}&-\overline{\alpha}_{-1}\alpha_{-2}
\end{matrix}
\right)
\end{equation*}
via a relation between the eigenfunctions of $\mathcal{C}_{-}$ and the associated right half line CMV matrix $\tilde{\mathcal{C}}_{+} = \mathcal{C}^{(0)}_{+} ( \{\tilde{\alpha}_{n}\}_{n\geq 0})$.

First, recall the Jitomirskaya-Last inquality for a right half line CMV matrix. With the solutions of \eqref{eq1} obeying \eqref{eq21} with $\omega=0$ and the local $\ell^2$ norms from Definition~\ref{d.subordinatesolution}, we have the following:

\begin{lemma}(\cite[Theorem 10.8.2]{S04}{\label{JLLe}})
For $z \in \partial \D$ and $r \in [0,1)$, define $x(r) \in (0,\infty)$ to be the unique solution of
\[
(1 - r) \| \varphi_{\cdot}(z) \|_{x(r)} \| \psi_{\cdot}(z) \|_{x(r)} = \sqrt{2}.
\]
Then
\begin{equation}{\label{eq16}}
A^{-1} |F_{+}(rz,0)| \leq \frac{\|\psi_{\cdot}(z)\|_{x(r)}}{\|\varphi_{\cdot}(z)\|_{x(r)}} \leq A |F_{+}(rz,0)|,
\end{equation}
where $A$ is a universal constant in $(1,\infty)$.
\end{lemma}

\begin{remark}{\label{re2}}
By Remark~\ref{re1}, ~(\ref{eq16}) is equivalent to
\begin{equation}{\label{eq22}}
A^{-1} |F_{+}(rz,0)| \leq \frac{\|p_{+}(z)\|_{x(r)}}{\|u_{+}(z)\|_{x(r)}} \leq A |F_{+}(rz,0)|,
\end{equation}
where $A \in (1,\infty)$ is a universal constant.
\end{remark}

Next we address the relation between $\mathcal{C}_{-}$ and $\tilde{\mathcal{C}}_{+}$. Let $J$ be the matrix with elements
\[
J_{i,j}=
\begin{cases}
1, \qquad &\textrm{ if } i = -j -1,\\
0, &\textrm{ otherwise},
\end{cases}
\]
for $i = -1, -2, -3, \ldots$ and  $j = 0, 1, 2, \ldots$.
Let $\tilde{J}$ be the matrix with elements
\[
\tilde{J}_{i,j}=
\begin{cases}
1, \qquad &\textrm{ if } i = -j - 1,\\
0, &\textrm{ otherwise},
\end{cases}
\]
for $i = 0, 1, 2, \ldots$ and $j = -1, -2, -3, \ldots$.

Define the operator $U : \ell^2(\Z)\rightarrow \ell^2(\Z)$ that maps $\ell^2 (\Z_-) \rightarrow \ell^2 (\Z_+)$ as follows:
\begin{equation*}
U=\left(
\begin{matrix}
\mathbf{0} & J \\
\tilde{J}      & \mathbf{0}
\end{matrix}
\right),
\end{equation*}
where  $\mathbf{0}$ is zero matrix. That is, $U \delta_{n} = \delta_{-n-1}$, for $n \in \N$.

A direct calculation implies
\begin{equation*}
U \mathcal{C}_{-} U^{*}=\left(
\begin{matrix}
-\overline{\alpha}_{-1}\alpha_{-2}& -\rho_{-2}\alpha_{-3}& \rho_{-2}\rho_{-3}& 0& 0&    \cdots&\\
\overline{\alpha}_{-1}\rho_{-2}& -\overline{\alpha}_{-2}\alpha_{-3}& \overline{\alpha}_{-2}\rho_{-3}& 0& 0&    \cdots&
\\
 0& -\rho_{-3}\alpha_{-4}& -\overline{\alpha}_{-3}\alpha_{-4}& -\rho_{-4}\alpha_{-5}&  \rho_{-4}\rho_{-5}& \cdots&
\\
 0& \rho_{-3}\rho_{-4}& \overline{\alpha}_{-3}\rho_{-4}& -\overline{\alpha}_{-4}\alpha_{-5}&  \overline{\alpha}_{-4}\rho_{-5}&  \cdots&
\\
\cdots&\cdots&\cdots&\cdots&\cdots&\cdots
\end{matrix}
\right).
\end{equation*}
Set $\tilde{\alpha}_{n}= - \overline{\alpha}_{-(n+2)}$. Then $\tilde{\rho}_{n}= \rho_{-(n+2)}$ and $\tilde{\mathcal{C}}_{+} = U \mathcal{C}_{-} U^{*}$.

For $\tilde{\mathcal{C}}_{+}$, denote $\tilde{\varphi}$ and $\tilde{\psi}$ to be the orthogonal polynomials and orthogonal polynomials of the second kind, respectively. Denote $u_{-}$ and $\tilde{u}_{+}$ ($p_-$ and $\tilde{p}_{+}$) to be the eigenfunctions for $\mathcal{C}_{-}$ and $\tilde{\mathcal{C}}_{+}$ respectively, and $v_{-}$ and $\tilde{v}_{+}$ ($q_-$ and $\tilde{q}_{+}$) to be the eigenfunctions for $\mathcal{C}^{\top}_{-}$ and $\tilde{\mathcal{C}}^{\top}_{+}$ respectively. Since $\tilde{\mathcal{C}}_{+} = U \mathcal{C}_{-} U^{*}$, $ \tilde{u}_{+}(n)= u_{-}(-(n+1))$ for $n\in\N$.

We have
\begin{equation*}
u_{-}(z, -n-1) = \tilde{u}_{+}(z , n) =
\begin{cases}
z^{\frac{-(n+1)}{2}} \tilde{\varphi}^{*}(z , n)\quad & n \textrm{ is }  odd,\\
-z^{\frac{-n}{2}} \tilde{\varphi}(z , n) &n \textrm{ is }  even,
\end{cases}
\end{equation*}
\begin{equation*}
v_{-}(z, -n-1) = \tilde{v}_{+}(z , n) =
\begin{cases}
-z^{\frac{-(n-1)}{2}} \tilde{\varphi}(z , n) \quad &n \textrm{ is }  odd,\\
z^{\frac{-n}{2}} \tilde{\varphi}^{*}(z , n) & n \textrm{ is }  even,
\end{cases}
\end{equation*}
\begin{equation*}
p_{-}(z,-n-1) = \tilde{p}_{+}(z , n) =
\begin{cases}
z^{\frac{-(n+1)}{2}} \tilde{\psi}^{*}(z , n)\quad & n \textrm{ is }  odd,\\
z^{\frac{-n}{2}} \tilde{\psi}(z , n) &n \textrm{ is }  even,
\end{cases}
\end{equation*}
\begin{equation*}
q_{-}(z, -n-1)=\tilde{q}_{+}(z , n) =
\begin{cases}
z^{\frac{-(n-1)}{2}} \tilde{\psi}(z , n) \quad &n \textrm{ is }  odd,\\
z^{\frac{-n}{2}} \tilde{\psi}^{*}(z , n) & n \textrm{ is }  even
\end{cases}
\end{equation*}
for $z \in \partial \D$.
Thus the initial conditions ~(\ref{eq9}) with $n_{0} = -1$ are equivalent to
\begin{equation}{\label{eq32}}
\left(
\begin{matrix}
-\tilde{\varphi}(z,0)\\
\tilde{\varphi}^{*}(z,0)
\end{matrix}
\right)
=
\left(
\begin{matrix}
-1\\
1
\end{matrix}
\right),
\qquad
\left(
\begin{matrix}
\tilde{\psi}(z,0)\\
\tilde{\psi}^{*}(z,0)
\end{matrix}
\right)
=
\left(
\begin{matrix}
1\\
1
\end{matrix}
\right).
\end{equation}

Since $z\in \partial \D$, $\|u_{-}(z)\|_{-x(r)-1} =\|\tilde{\varphi}(z)\|_{x(r)}$ and $\|p_{-}(z)\|_{-x(r)-1}  =\|\tilde{\psi}(z)\|_{x(r)}$, where $x(r)$ is as in Theorem~\ref{JLLe}. Due to Lemma~\ref{le01}, there must then be a unique $\tilde{F}_{+}(z)$ such that
$$
\left(\begin{matrix}
-\tilde{\varphi}(z,n) + \tilde{F}_{+}(z) \tilde{\psi}(z, n)\\
\tilde{\varphi}^{*}(z,n) + \tilde{F}_{+}(z) \tilde{\psi}^{*}(z, n)
\end{matrix}
\right) \in \ell^{2}(\mathbb{Z}_{+}).
$$

Due to the unitarity of $U$, $F_{-}(z,-1) = - \tilde{F}_{+}(z,0)$, where $\tilde{F}_{+}(z,0)$ is the Carath\'{e}odory function for $\tilde{\mathcal{C}}_{+}$. Hence, the Jitomirskaya-Last inequality holds for $\mathcal{C}_{-}$. For $z \in \partial \D$ and $r\in [0,1)$, define $x_1(r)\in (-\infty,-1)$ to be the unique solution of
\[
    (1-r) \|u_{-}(z)\|_{x_1(r)} \|p_{-}(z)\|_{x_1(r)}  = \sqrt{2}.
\]
Then
\begin{equation}{\label{eq17}}
    A^{-1} |F_{-}(rz,-1)| \leq \frac{\|u_{-}(z)\|_{x_1(r)}}{\|p_{-}(z)\|_{x_1(r)}} \leq A |F_{-}(rz,-1)|,
\end{equation}
where $A$ is a universal constant in $(1,\infty)$.

Next, we will extend the Jitomirskaya-Last inequality, which holds for the boundary condition $\varphi(z,0) = 1$, to a general boundary condition of the form
\begin{equation}{\label{eq20}}
\varphi(z,0) (\cos\omega - i \sin\omega) - \varphi^{*}(z,0) (\cos\omega + i\sin\omega) = 0.
\end{equation}

Given $z \in \D$ and $\omega \in [0,\pi)$, let $\left(\begin{matrix}\varphi_{\omega}\\ \varphi^{*}_{\omega}\end{matrix}\right)$ and $\left(\begin{matrix}\psi_{\omega}\\ -\psi^{*}_{\omega}\end{matrix}\right)$ denote the solutions of ~\eqref{eq19} obeying ~\eqref{eq21}.
Thus, $\left(\begin{matrix}\varphi_{\omega}\\ \varphi^{*}_{\omega}\end{matrix}\right)$ obeys the boundary condition~\eqref{eq20} and $\left(\begin{matrix}\psi_{\omega}\\ -\psi^{*}_{\omega}\end{matrix}\right)$ obeys the orthogonal boundary condition.

Define $u_{\omega} (z,n)$ and $p_{\omega} (z,n)$to be the solutions of~\eqref{eq1}, subjected to the boundary condition~\eqref{eq21}. For $r\in[0,1)$, define $x(r)$ to be unique solution of
\begin{equation}{\label{eq29}}
(1 - r) \|u_{\omega}(z)\|_{x(r)} \|p_{\omega}(z)\|_{x(r)} = \sqrt{2}.
\end{equation}
By \cite[Theorem 2.18]{GZ06}, there are a unique $F_{+}^{\omega}(z,0)$ such that
\begin{equation*}
\xi^{\omega}_{+}(z,n,0) = p_{\omega}(z,n) + F_{+}^{\omega}(z,0) u_{\omega}(z,n),
\end{equation*}
is $\ell^2$ at infinity and a unique $M_{-}^{\omega}(z,0)$ such that
\begin{equation*}
\xi^{\omega}_{-}(z,n,0) = p_{\omega}(z,n) + M_{-}^{\omega}(z,0) u_{\omega}(z,n),
\end{equation*}
is $\ell^2$ at $-\infty$.
By~\eqref{eq12} and ~\eqref{eq27}, we have
\begin{align*}
\xi_{+}^{\omega}(z,0,0) & = \psi_{\omega}(z,0) + F_{+}(z,0) \varphi_{\omega}(z,0), \\ \xi_{-}^{\omega}(z,0,0) & = \psi_{\omega}(z,0) + M_{-}(z,0) \varphi_{\omega}(z,0).
\end{align*}
Define
\begin{align*}
\xi_{+}^{\omega,*}(z,0,0) & = -\psi_{\omega}^{*}(z,0) + F_{+}(z,0) \varphi_{\omega}(z,0), \\ \xi_{-}^{\omega,*}(z,0,0) & = -\psi_{\omega}^{*}(z,0) + M_{-}(z,0) \varphi_{\omega}^{*}(z,0).
\end{align*}
With these definitions the following generalization of Lemma~\ref{le01} holds:

\begin{lemma}{\label{le3}}
For $z \in \partial \D$, define $x(r) \in (0,\infty)$ to be unique solution of
\begin{equation*}
(1 - r) \|u_{\omega}(z)\|_{x(r)} \|p_{\omega}(z)\|_{x(r)} = \sqrt{2}.
\end{equation*}
Then we have
\begin{equation}{\label{eq23}}
A_{1}^{-1} |F^{\omega}_{+}(rz,0)| \leq \frac{\|p_{\omega}(z)\|_{x(r)}}{\|u_{\omega}(z)\|_{x(r)}} \leq A_{1} |F^{\omega}_{+}(rz,0)|,
\end{equation}
where $A_1$ is a universal constant in $(1,\infty)$.
Similarly, define $x_1(r) \in (-\infty,-1)$ to be unique solution of
\begin{equation*}
(1 - r) \|u_{\omega}(z)\|_{x_1(r)} \|p_{\omega}(z)\|_{x_1(r)} = \sqrt{2}.
\end{equation*}
Then we have
\begin{equation}{\label{JLineq.lefthalfline.omega}}
A_{2}^{-1} |F^{\omega}_{-}(rz,0)| \leq \frac{\|p_{\omega}(z)\|_{x_1(r)}}{\|u_{\omega}(z)\|_{x_1(r)}} \leq A_{2} |F^{\omega}_{-}(rz,0)|,
\end{equation}
where $A_2$ is a universal constant in $(1,\infty)$.
\end{lemma}

\begin{proof}
In order to follow the proof of \cite[Theorem 10.8.2]{S05}, denote $T_n$ as :
\begin{equation*}
T_{n} = \frac{1}{2}
\left(
\begin{matrix}
e^{-i\omega} (\varphi_{\omega}(n) + \psi_{\omega}(n))        & e^{i\omega} (\varphi_{\omega}(n) - \psi_{\omega}(n))\\
e^{-i\omega} (\varphi^{*}_{\omega}(n) - \psi^{*}_{\omega}(n)) &  e^{i\omega} (\varphi^{*}_{\omega}(n) + \psi^{*}_{\omega}(n))
\end{matrix}
\right).
\end{equation*}
This is from \cite[(3.2.27)]{S04} with the boundary conditions
	$\left(\begin{matrix} e^{i\omega} \\ e^{-i\omega} \end{matrix}\right)$, $\left(\begin{matrix} e^{i\omega} \\ -e^{-i\omega} \end{matrix}\right)$. By \cite[(3.2.28)]{S04}, $\det{T_n} =  z^n$. We have
	\begin{equation*}
T_{l}^{-1} = \frac{1}{2z^l}
\left(
\begin{matrix}
   e^{i\omega} (\varphi^{*}_{\omega}(l) + \psi^{*}_{\omega}(l))     & - e^{i\omega} (\varphi_{\omega}(l) - \psi_{\omega}(l))\\
 - e^{-i\omega} (\varphi^{*}_{\omega}(l) - \psi^{*}_{\omega}(l)) &  e^{-i\omega} (\varphi_{\omega}(l) + \psi_{\omega}(l))
\end{matrix}
\right).
\end{equation*}
A direct calculation shows that $T_{n\leftarrow l} = T_nT_l^{-1}$ which is same as $T_{n\leftarrow l}$ in \cite[(10.8.8)]{S05}. The remaining proof is same as the proof of \cite[Theorem 10.8.2]{S05}. We can then conclude that (\ref{eq23}) holds.

\end{proof}

\begin{lemma}{\label{le4}}
Let $\theta\in [0, 2\pi)$ be given.
\begin{enumerate}
\item One has $\lim_{r\uparrow 1} F_{+}(re^{i\theta},0) = - i \cot \omega$ for some $\omega \in [0,\pi)$ if and only if $u_{\omega}$ is subordinate at $+\infty$.
\item One has $\lim_{r\uparrow 1} M_{-}(re^{i\theta},0) = - i \cot \omega$ for some $\omega \in [0,\pi)$ if and only if $u_{\omega}$ is subordinate at $-\infty$.
\item The difference equation ~(\ref{eq1}) enjoys a subordinate solution at $+ \infty$ if and only if $\lim_{r\uparrow 1} F_{+}(re^{i\theta},0) \in i(\R\bigcup \set{\infty})$.
\item The difference equation ~(\ref{eq1}) enjoys a subordinate solution at $- \infty$ if and only if $\lim_{r\uparrow 1} M_{-}(re^{i\theta},0) \in i(\R\bigcup \set{\infty})$.
\item The difference equation ~(\ref{eq1}) enjoys a solution that is subordinate at both $\pm \infty$ if and only if $\lim_{r\uparrow 1} F_{+}(re^{i\theta},0) = \lim_{r\uparrow 1} M_{-}(re^{i\theta},0)$.
\end{enumerate}
\end{lemma}

\begin{proof}

(1) Consider the $m$-function for OPUC of the form
$$
m_{0}^{+}(re^{i\theta}) = \frac{\xi^{0,*}_{+}(re^{i\theta},0,0)}{\xi^{0}_{+}(re^{i\theta},0,0)}.
$$
By Lemma~\ref{le2},
\begin{equation*}
m_{0}^{+}(re^{i\theta}) = \frac{\xi_{+}^{\omega,*}(re^{i\theta},0,0)}{\xi_{+}^{\omega}(re^{i\theta},0,0)}
\end{equation*}
implies
\begin{equation*}
\frac{-\psi^{*}(re^{i\theta},0) + F_{+}(re^{i\theta},0) \varphi^{*}(re^{i\theta},0)}{\psi(re^{i\theta},0) + F_{+}(re^{i\theta},0) \varphi(re^{i\theta},0)} = \frac{-\psi_{\omega}^{*}(re^{i\theta},0) + F^{\omega}_{+}(re^{i\theta},0) \varphi^{*}_{\omega}(re^{i\theta},0)}{\psi_{\omega}(re^{i\theta},0) + F^{\omega}_{+}(re^{i\theta},0) \varphi_{\omega}(re^{i\theta},0)}.
\end{equation*}
For simplicity, write $\varphi_{\omega}(re^{i\theta},0)$ and $\psi_{\omega}(re^{i\theta},0)$ as $\varphi_{\omega}$ and $\psi_{\omega}$, respectively. It follows that
$$
\frac{-1 + F_+(re^{i\theta},0)}{1+F_+(re^{i\theta},0)} = \frac{-\psi_{\omega}^{*} + F_{+}^{\omega}(re^{i\theta},0) \varphi_{\omega}^{*}}{\psi_{\omega}+F_{+}^{\omega}(re^{i\theta},0)\varphi_{\omega}},
$$
and then
$$
F_{+}^{\omega}(re^{i\theta},0) = \frac{-\psi^{*}_{\omega} +\psi_{\omega} - F_{+}(re^{i\theta},0)(\psi_{\omega}^{*}+\psi_{\omega})}{-(\varphi_{\omega}+\varphi_{\omega}^{*}) + F_{+}(re^{i\theta},0)(\varphi_{\omega} - \varphi_{\omega}^{*})}.
$$

By~\eqref{eq21}, we have
\begin{equation}{\label{eq33}}
F^{\omega}_{+}(re^{i\theta},0) =  \frac{i \sin \omega - F_{+}(re^{i\theta},0) \cos \omega}{- \cos\omega + i F_{+}(re^{i\theta},0) \sin \omega}.
\end{equation}

We see that $\lim_{r\uparrow 1}|F^{\omega}_{+}(r e^{i\theta},0)| = \infty$ if and only if $\lim_{r\uparrow 1}F_{+}(r e^{i\theta},0) =  -i \cot \omega$. By Lemma~\ref{le3}, $\lim_{r\uparrow 1}|F^{\omega}_{+}(r e^{i\theta},0)| = \infty$ if and only if $u_{\omega}$ is subordinate at $+\infty$. Putting the two equivalences together, (1) follows.

(2) Consider $m$-function for left half line CMV matrix $\tilde{\CC}_{+}$ as $\tilde{m}^{-}(re^{i\theta}) = \frac{\xi^{0,*}_{-}(re^{i\theta},0,0)}{\xi^{0}_{-}(re^{i\theta},0,0)}$. It implies
\begin{equation*}
\frac{-\psi^{*}(re^{i\theta},0) + M_{-}(re^{i\theta},0) \varphi^{*}(re^{i\theta},0)}{\psi(re^{i\theta},0) + M_{-}(re^{i\theta},0) \varphi(re^{i\theta},0)} = \frac{-\psi_{\omega}^{*}(re^{i\theta},0) + M^{\omega}_{-}(re^{i\theta},0) \varphi^{*}_{\omega}(re^{i\theta},0)}{\psi_{\omega}(re^{i\theta},0) + M^{\omega}_{-}(re^{i\theta},0) \varphi_{\omega}(re^{i\theta},0)}.
\end{equation*}
A direct calculation gives
\begin{equation}{\label{eq34}}
M^{\omega}_{-}(re^{i\theta},0) =  \frac{i \sin \omega - M_{-}(re^{i\theta},0) \cos \omega}{- \cos\omega + i M_{-}(re^{i\theta},0) \sin \omega}.
\end{equation}
Since the $\ell^2$ solution is unique up to a non-zero constant $C$, by \cite[Remark 2.19]{GZ06}, we have
\begin{equation*}
C \xi_{-}(z,0,0) = p_{-}(z,0,0) + F_{-}(z,0) u_{-}(z,0,0),
\end{equation*}
which implies that
\begin{equation}{\label{eq30}}
C (1 + M_{-}(z,0)) = 1 + F_{-}(z,0),
\end{equation}
where $F_{-}(z,0)$ is the Carath\'{e}odory function for $C_{-}$ acting on $\ell^2([0,-\infty)\bigcap\Z)$ with $|\alpha_{0}| = 1$.
Thus, $M_{-}(z,0) = 1 - C^{-1}(1 + F_{-}(z,0))$, which means $\lim_{r\uparrow 1}|M_{-}(re^{i\theta},0)| = \infty$ if and only if
$\lim_{r\uparrow 1}|F_{-}(re^{i\theta},0)| = \infty$.
We see that $\lim_{r\uparrow 1}|M^{\omega}_{-}(r e^{i\theta},0)| = \infty$ if and only if $\lim_{r\uparrow 1}M_{-}(r e^{i\theta},0) =-i \cot \omega$. By Jitomirskaya-Last inequalities, $\lim_{r\uparrow 1}|F^{\omega}_{-}(r e^{i\theta},0)| = \infty$ if and only if $u_{\omega}$ is subordinate at $-\infty$. Putting the three equivalences together, (2) follows.

We observe that (3) and (4) follow immediately from (1) and (2). Recall now that $F_{+}(re^{i\theta},0)$ is a Carath\'{e}odory function and $M_-(re^{i\theta},0)$ is an anti-Carath\'{e}odory function. Thus (5) follows immediately from (3) and (4) upon noting that $\lim_{r\uparrow 1} F_{+}(re^{i\theta},0) = \lim_{r\uparrow 1} M_{-}(re^{i\theta},0)$ forces the common limit to belong to $i(\R\bigcup \set{\infty})$.
\end{proof}

\section{Proof of Theorem~\ref{t.main1rev}}\label{s.5}

In this section we prove Theorem~\ref{t.main1rev}. We will make use of the preparatory work in the previous sections, culminating in Lemmas~\ref{le3} and \ref{le4}.

\begin{proof}[Proof of Theorem~\ref{t.main1rev}]
(a) As discussed in Remark~\ref{r.ppstatement}, this statement is well known and a proof can be given by arguments similar to those in the proof of \cite[Theorem~7.27.(a)]{W80}.

\bigskip

Before turning our attention to the statements (b) and (c) in Theorem~\ref{t.main1rev}, let us recall that $F$ is the Borel transform of the canonical spectral measure $\Lambda$ and it can be expressed as follows,
\begin{align}
F (z) &=  1 + 2 z (G_{00} (z) + G_{11} (z)) \notag\\
      &=  -1 + \frac{(\overline{\alpha}_{0} + 2 z + \alpha_{0} z^2) + (\alpha_{0} z^2 - \overline{\alpha}_{0}) (M_{-}(z,0) + F_{+}(z,0))}{ \rho_{0}^2 z (F_{+}(z,0) - M_{-}(z,0)) } \notag \\
      &\quad +\frac{(\overline{\alpha}_{0} - 2 z + \alpha_{0} z^2) M_{-}(z,0) F_{+}(z,0)} { \rho_{0}^2 z (F_{+}(z,0) - M_{-}(z,0)) }   \label{eq5}\\
      &=  -1 + \frac{\frac{(\overline{\alpha}_{0} + 2 z + \alpha_{0} z^2)}{F_{+}(z,0) M_{-}(z,0)} +  (\alpha_{0} z^2 - \overline{\alpha}_{0}) (\frac{1}{F_{+}(z,0)} + \frac{1}{M_{-}(z,0)}) +(\overline{\alpha}_{0} - 2 z + \alpha_{0} z^2) } { \rho_{0}^2z (\frac{1}{M_{-}(z,0)} - \frac{1}{F_{+}(z,0)}) }. \label{eq6}
\end{align}
From \cite[Section 1.3.5]{S04}, recall that $\Lambda_{\sing}$ is supported on
\begin{equation*}
S_{\Lambda} = \set{\theta : \lim_{r\uparrow 1} \Re F(r e^{i\theta}) = \infty},
\end{equation*}
and an essential support of $\Lambda_{\ac}$ is given by
\begin{equation*}
\CA_{\Lambda} = \set{\theta : 0 < \lim_{r\uparrow 1} \Re F(r e^{i\theta}) < \infty}.
\end{equation*}
With these preliminaries out of the way, we can now address the items (b) and (c) of Theorem~\ref{t.main1rev}.

\bigskip

(b)
It suffices to show that $\Lambda_{\sing} (S_{\Lambda} \backslash S) = 0$. Consider $\theta \in S_{\Lambda}$, that is, $\lim_{r\uparrow 1} \Re F(r e^{i\theta})$ exists and is infinite. In particular, we also have $\lim_{r\uparrow 1} | F(r e^{i\theta})| = \infty$. There are three cases:
\begin{enumerate}
	\item $\lim_{r\uparrow 1} F_{+}(r e^{i\theta},0)$ and $\lim_{r\uparrow 1} M_{-}(r e^{i\theta},0)$ both exist (with $\infty$ being an admissible limit).
	\item Exactly one of $\lim_{r\uparrow 1} F_{+}(r e^{i\theta},0)$ and $\lim_{r\uparrow 1} M_{-}(r e^{i\theta},0)$exists.
	\item Neither of them exists.
\end{enumerate}

We will show that the $\theta$'s that are in case (1) belong to $S$, case (2) is impossible, and the $\theta$'s in case (3) have zero measure with respect to $\Lambda_{s}$. Combining these three statements, we obtain the desired $\Lambda_{\sing}(S_{\Lambda} \backslash S) = 0$.

\textbf{Case (1). } Suppose the limits $\lim_{r\uparrow 1} F_{+}(r e^{i\theta},0)$ and $\lim_{r\uparrow 1} M_{-}(r e^{i\theta},0)$ both exist (with $\infty$ being an admissible limit). By ~(\ref{eq5}), as $r\uparrow 1$, we must have
\begin{equation}{\label{eq7}}
|F_{+} (r e^{i\theta},0) - M_{-} (r e^{i\theta},0)| \rightarrow 0 ,
\end{equation}
or
\begin{align}
&|(\overline{\alpha}_{0} + 2 z + \alpha_{0} z^2) +  (\alpha_{0} z^2 - \overline{\alpha}_{0}) (M_{-}(r e^{i\theta},0) + F_{+}(r e^{i\theta},0)) + \label{eq18}\\
+&(\overline{\alpha}_{0} - 2 z + \alpha_{0} z^2) M_{-} (r e^{i\theta},0)F_{+}(r e^{i\theta},0)| \rightarrow \infty. \notag
\end{align}

Let us consider the first case, $|F_{+} (r e^{i\theta},0) - M_{-} (r e^{i\theta},0)| \rightarrow 0$, which implies that $\lim_{r\uparrow 1}  F_{+}(r e^{i\theta},0) = \lim_{r\uparrow 1} M_{-}(r e^{i\theta},0)$. If $\lim_{r\uparrow 1}  |F_{+}(r e^{i\theta},0)| = \lim_{r\uparrow 1} |M_{-}(r e^{i\theta},0)|$ is finite, by part (5) of Lemma~\ref{le4}, there is a solution $u_{\omega}$ that is subordinate at both $\pm \infty$ with non-zero $\omega$. If $\lim_{r\uparrow 1}  |F_{+}(r e^{i\theta},0)| = \lim_{r\uparrow 1} |M_{-}(r e^{i\theta},0)|$ is infinite, $\lim_{r\uparrow 1}|F_{-} (r e^{i\theta},0)|$ is infinite by \eqref{eq30}. Thus, $u_{+}(e^{i\theta})$ is a subordinate solution at $\infty$ and $u_{-}(e^{i\theta})$ is a subordinate solution at $-\infty$. From Lemma~\ref{le4} (1) and (2), we have $F_{+} (r e^{i\theta},0)= -i\cot\omega_{1}$ and $M_{-} (r e^{i\theta},0) = -i\cot\omega_{2}$, so $\omega_{1}=\omega_{2}=0$, which means there is a subordinate solution $u_{0}(e^{i\theta})$ at both $\pm\infty$.

Now, we consider the second case, which means $|M_{-} (r e^{i\theta},0) F_{+} (r e^{i\theta},0)| \rightarrow \infty$, $|F_{+} (r e^{i\theta},0)| \rightarrow \infty$ or $|M_{-} (r e^{i\theta},0)| \rightarrow \infty$.
In this case, if  one of $|M_{-} (r e^{i\theta},0)|$ and $|F_{+} (r e^{i\theta},0)|$ goes to infinity while the other one goes to a constant, then $|F(r e^{i \theta})|$ cannot go to infinity.

Therefore, both $|M_{-} (r e^{i\theta},0)|$ and $|F_{+} (r e^{i\theta},0)|$ go to infinity. As in the first case, there is a subordinate solution $u_{0}(e^{i\theta})$ at both $\pm\infty$.

\textbf{Case (2). } Suppose exactly one of $\lim_{r\uparrow 1} F_{+}(r e^{i\theta},0)$ and $\lim_{r\uparrow 1} M_{-}(r e^{i\theta},0)$ exists, say $\lim_{r\uparrow 1} F_{+}(r e^{i\theta},0)$.
As in case (1), one of \eqref{eq7} and \eqref{eq18} must hold true.
First, consider the case where \eqref{eq7} holds. It implies that if $\lim_{r\uparrow 1} | F_{+}(r e^{i\theta},0)|$ exists, then $\lim_{r\uparrow 1} | M_{-}(r e^{i\theta},0)|$ also exists, which is a contradiction.
Now, consider the case where \eqref{eq18} holds, which implies $\lim_{r\uparrow 1} | F_{+}(r e^{i\theta},0)| = \infty$.
By ~(\ref{eq6}), $| M_{-}(r e^{i\theta},0)|$ must go to infinity as $r\uparrow 1$, which is a contradiction.
Thus, case (2) is impossible.

\textbf{Case (3). } Suppose that neither $\lim_{r\uparrow 1} F_{+}(r e^{i\theta},0)$ nor $\lim_{r\uparrow 1} M_{-}(r e^{i\theta},0)$ exists. Therefore, $\liminf_{r\uparrow 1} F_{+}(r e^{i\theta},0)$ and $\liminf_{r\uparrow 1} M_{-}(r e^{i\theta},0)$ are finite. We can thus choose a sequence $r_{n}\uparrow 1$ in such a way that the limits
\begin{equation*}
\ell_{+}:=\lim_{r_{n}\uparrow 1} F_{+}(r_n e^{i\theta},0) \textrm{ and } \ell_{-}= \lim_{r_{n}\uparrow 1} M_{-}(r_n e^{i\theta},0)
\end{equation*}
exist, and $\ell_+$ is finite.
 Following a similar argument to that in Case 1, since $\lim_{r \uparrow 1}|F(r e^{i\theta})| = \infty$, $\ell_{-} = \ell_{+} \in i \R $, which means that all accumulation points of $\set{F_{+}(r e^{i\theta},0)}_{0<r<1}$ and $\set{M_{-}(r e^{i\theta},0)}_{0<r<1}$ are either $0$ or purely imaginary numbers. Thus $\Re \ell_{+} = 0$.

Since $M_{-}(r e^{i\theta},0)$ is an anti-Caratheodory function, it is analytic in $\D$. If $ia, ib \in i \R$ are two different accumulation points of $M_{-}(r e^{i\theta},0)$, by the intermediate value theorem and the fact that all accumulation points are imaginary, we can find a sequence of $\set{r^{\prime}_n}$ such that $M_{-}(r^{\prime}_n e^{i\theta},0)$ goes to $ic$ for any $c \in (a,b)$ as $r^{\prime}_n \uparrow 1$. We thus observe that since the limits do not exist and the real parts of accumulation points are zero,  $M_{-}(r e^{i\theta},0)$ has uncountably many accumulation points.
	
So we can choose $r_n$ and $t_n$ in such a way that
\begin{align*}
\lim_{r_n \uparrow 1} M_{-} (r_n e^{i\theta},0) = \ell_1,\quad & \lim_{t_n \uparrow 1} M_{-} (t_n e^{i\theta},0) = \ell_2\\
\lim_{r_n \uparrow 1} F_{+} (r_n e^{i\theta},0) = \ell_1,\quad & \lim_{t_n \uparrow 1} F_{+} (t_n e^{i\theta},0) = \ell_2,
\end{align*}
where $\ell_{1} \neq \ell_{2}$ and $\ell_{1},\ell_{2} \in i\R$.
Let $\Lambda_{0}$ denote the $\delta_0$-spectral measure of $\mathcal{E}$, given by
\begin{equation*}
\langle \delta_0, (\mathcal{E}+ z I)(\mathcal{E} - zI)^{-1}\delta_0 \rangle = \int_{\partial \D} \frac{e^{i\theta} + z}{e^{i\theta} -z} d \Lambda_{0} : = M_{00}, \quad z\in \C\backslash \partial \D.
\end{equation*}
By \cite[(3.25)]{GZ06}, $M_{00}$ is of the form
\begin{equation*}
M_{00}(z) = 1 + \frac{[1- \alpha_0 - (1 + \alpha_0) F_{+}(z,0)] [1 - \overline{\alpha}_{0} + (1 + \overline{\alpha}_{0} M_{-}(z,0))]}
{\rho_0^2 (F_{+}(z,0) - M_{-}(z,0))}.
\end{equation*}
Defining $\Lambda_1$ and $M_{11}$ similarly, the canonical spectral measure is $\Lambda = \Lambda_0 + \Lambda_1$.
Since $\Lambda_0 \ll \Lambda$ and $F(z) = \int \frac{e^{i\theta} + z}{e^{i\theta} - z} d\Lambda(\theta)$, the corresponding Radon-Nikodym derivative satisfies
\begin{equation}{\label{eq24}}
\frac{d\Lambda_0}{d\Lambda} (\theta)= \lim_{r\uparrow 1}  \frac{M_{00}(r e^{i\theta})}{F(r e^{i\theta})},
\end{equation}
for $\Lambda_s$-almost every $\theta$.

	This follows from Poltoratskii's theorem \cite{P94}, \cite{JL04}. We are grateful to Jake Fillman for the idea to use Poltoratskii's theorem to simplify this part of the proof.  The version we use can be found in \cite[Remark after Proposition 3.1]{S06}, which says that for any complex Borel measure $\mu$ on $\partial\D$ and any $g\in L^{1}(\partial \D, d\mu)$ we have, for almost any $e^{i \theta^{\prime}}$ with respect to $d\mu_{\sing}$ (but not for $d\mu_{\ac}$), that
	$$
	\lim_{r\uparrow 1}\frac{\int_{\partial \D} \frac{e^{i\theta} + r e^{i\theta^{\prime}}}{e^{i\theta} -r e^{i\theta^{\prime}}}f(\theta) d \mu(\theta)}{\int_{\partial \D} \frac{e^{i\theta} + r e^{i\theta^{\prime}}}{e^{i\theta} -r e^{i\theta^{\prime}}} d \mu(\theta)} = f(\theta^{\prime}).
	$$

Due to our choices of $r_n, t_n$ and $\ell_{1} \neq \ell_{2}$, we obtain
\begin{equation*}
\lim_{n \rightarrow \infty}  \frac{M_{00}(r_n e^{i\theta})}{F(r_n e^{i\theta})} \neq
\lim_{n \rightarrow \infty}  \frac{M_{00}(t_n e^{i\theta})}{F(t_n e^{i\theta})}.
\end{equation*}
Thus, the $\theta$'s in case (3) such that ~(\ref{eq24}) fails and the set has zero $\Lambda_s$-measure.

\bigskip

(c) Since $\Lambda_{\ac}$ gives zero weight to sets of zero Lebesgue measure and $N_{\Lambda}$ is an essential support, it suffices to show that
\begin{equation}{\label{eq11}}
\Leb(\CA \backslash \CA_{\Lambda}) = 0
\end{equation}
and
\begin{equation}{\label{eq14}}
\Leb(\CA_{\Lambda} \backslash \CA) = 0.
\end{equation}
Assume that $\lim_{r \uparrow 1} F (r e^{i\theta})$, $\lim_{r \uparrow 1} F_{+} (r e^{i\theta},0)$ and $\lim_{r \uparrow 1} M_{-}(r e^{i\theta},0)$ all have finite boundary values as $r\uparrow 1$ for all $z=e^{i\theta}$ in question.

To prove (\ref{eq11}), we consider $\theta \in \CA_{\Lambda}$ for which $\lim_{r \uparrow 1}| F (r e^{i\theta})|$, $\lim_{r \uparrow 1} |F_{+} (r e^{i\theta},0)|$ and $\lim_{r \uparrow 1} |M_{-}(r e^{i\theta},0)|$ exist and are finite, and show that $z=e^{i\theta} \in \CA$.

Let us first consider the possibility that $\lim_{r \uparrow 1} F_{+} (r e^{i\theta},0)$ and $\lim_{r \uparrow 1} M_{-}(r e^{i\theta},0)$ are both purely imaginary. Specifically, let $\lim_{r \uparrow 1} F_{+} (r e^{i\theta},0) = i a$ and $\lim_{r \uparrow 1} M_{-} (r e^{i\theta},0) = i b$, where $a\neq b \in\mathbb{R}$.
Since $F(r e^{i\theta}) = 1 + 2 z (G_{00} +G_{11})$, we firstly consider $1 + 2zG_{00}$.
\begin{align*}
1 + 2zG_{00} = & \frac{i(a - b) - (-1 + ia)(1 + ib)}{i (a - b)}\\
             = & \frac{1+ab}{i (a -b)},
\end{align*}
which is a purely imaginary number.
Consider $2zG_{11}$, which is
\begin{align*}
2zG_{11} & = - \frac{[z + \overline{\alpha}_0 + ib (z-\overline{\alpha}_{0})][-1-\alpha_0 z + ia(1 - \alpha_0 z)]}{i \rho_0^2 z (a - b)}\\
         & = \frac{(z + \overline{\alpha}_0) (1 + \alpha_0 z) + ab(z-\overline{\alpha}_{0}) (1 - \alpha_0 z)}{i \rho_0^2 z (a - b)} \\
         & \quad + \frac{i [b(z-\overline{\alpha}_{0})(1+\alpha_{0}z)- a (z + \overline{\alpha}_0)(1-\alpha_0 z)]}{i \rho_0^2 z (a - b)}\\
         & = \frac{z+\overline{\alpha}_0 +\alpha_{0}z^2 + |\alpha_{0}|^2z + abz-\overline{\alpha}_{0}ab - \alpha_0 ab z^2 + ab|\alpha_0|^2 z}{i\rho_0^2 z (a - b)} \\
         & \quad + \frac{i[bz - b\overline{\alpha}_0 + b\alpha_0 z^2 - b|\alpha_0|^2 z - az - a \overline{\alpha}_0 + a \alpha_0 z^2 + a|\alpha_0|^2 z]}{i\rho_0^2 z (a - b)}\\
         & = \frac{1+\overline{\alpha}_0 z^{-1} +\alpha_{0}z + |\alpha_{0}|^2 + ab -\overline{\alpha}_{0}abz^{-1} - \alpha_0 ab z + ab|\alpha_0|^2 }{i\rho_0^2  (a - b)} \\
         & \quad + \frac{i[b - b\overline{\alpha}_0z^{-1} + b\alpha_0 z - b|\alpha_0|^2  - a - a \overline{\alpha}_0z^{-1} + a \alpha_0 z + a|\alpha_0|^2 ]}{i\rho_0^2 (a - b)}.
\end{align*}
Thus,
\begin{align*}
\Re(-2\rho_0^2 (a-b)z G_{11}) & = \Re(i \overline{\alpha}_0 z^{-1} + i\alpha_{0} z - iab\overline{\alpha}_{0}z^{-1} - i ab\alpha_0 z + (a-b)\rho_0^2\\
                              & \qquad + b \overline{\alpha}_0 z^{-1} - b\alpha_{0}z + a\overline{\alpha}_0 z^{-1} - a\alpha_0z)\\
                              & = (1 - ab) \Re(i \overline{\alpha}_0 z^{-1} + i \alpha_0 z) +(a + b) \Re(\overline{\alpha}_0z^{-1} - \alpha_0 z) \\
                              & \qquad + (a-b) \rho_0^2.
\end{align*}
With $z=re^{i\theta}$, we have
$$
\Re(i\overline{\alpha}_0 r^{-1}e^{-i\theta} + i \alpha_0 re^{i\theta}) \rightarrow \Re(i (\overline{\alpha_0 e^{i\theta}}+\alpha_0 e^{i\theta}))=0
$$
and
$$
\Re(\overline{\alpha}_0r^{-1}e^{-i\theta} - \alpha_0 re^{i\theta}) \rightarrow \Re(\overline{\alpha_0 e^{i\theta}}- \alpha_0 e^{i\theta}) =0,
$$
as $r\uparrow 1$.

Putting these together, we have $\lim_{r \uparrow 1} \Re F (r e^{i\theta}) = -1$, which is impossible. If $a=b$, $|\lim_{r \uparrow 1} F (r e^{i\theta})|=\infty$, which is also impossible. We conclude that $\lim_{r \uparrow 1} F_{+} (r e^{i\theta},0)$ and $\lim_{r \uparrow 1} M_{-}(r e^{i\theta},0)$ cannot both be purely imaginary.

Without loss of generality, we assume that $\lim_{r \uparrow 1} F_{+} (r e^{i\theta},0)$ is not a purely imaginary number, which implies there is no subordinate solution at $+\infty$ by Lemma \ref{le4} (3). Thus, $z=e^{i\theta} \in \CA_{+}$. Therefore, $z=e^{i\theta}\in\CA$.

To prove \eqref{eq14}, we consider $z=e^{i\theta} \in \CA$ for which $\lim_{r \uparrow 1} |F (r e^{i\theta})|$, $\lim_{r \uparrow 1} |F_{+} (r e^{i\theta},0)|$ and $\lim_{r \uparrow 1} |M_{-}(r e^{i\theta},0)|$ exist and are finite, and show that $\theta \in \CA_{\Lambda}$. Without loss of generality, assume that $z=e^{i\theta}\in \CA_{+}$, so \eqref{eq1} has no solution that is subordinate at $+\infty$. Then Lemma \ref{le4} implies that $\lim_{r \uparrow 1} F_{+} (r e^{i\theta},0)$ cannot be purely imaginary, which implies that $0<\lim_{r \uparrow 1} \Re F_{+} (r e^{i\theta},0)<\infty$. Since $\lim_{r \uparrow 1} M_{-} (r e^{i\theta},0)$ also exists and must belong to $\mathbb{C}_{\ell} = \set{z\in\mathbb{C}|\Re(z)<0}$, we have $0<\lim_{r \uparrow 1} \Re F (r e^{i\theta})<\infty$. Therefore, $\theta\in\CA_{\Lambda}$.
\end{proof}

\section{Proof of the Corollaries}\label{s.6}

In this section we prove Corollaries~\ref{c.bddsol} and \ref{c.ddvcbddc}.

\begin{proof}[Proof of Corollary~\ref{c.bddsol}]
Suppose that $z \in \mathcal{B}_+$, that is, $z \in \partial \D$ is such that
$$
\sup_{n \in \Z_\pm} \|A(n,z)\| < \infty.
$$
By an extension of the discussion in the proof of \cite[Corollary~10.8.4]{S05}, it follows from the Jitomirskaya-Last inequality that \eqref{eq1} has no solution that is subordinate at $+ \infty$, and hence $z \in \mathcal{A}_+$.
Specifically, due to \eqref{eq12}--\eqref{eq28}, we have
\begin{align*}
&|u_+(z,n)|=|\varphi_n(z)| \textrm{ and }|p_+(z,n)|=|\psi_n(z)|\quad \textrm{ when } n \textrm{ is even,} \\
&|v_+(z,n)|=|\varphi_n(z)| \textrm{ and }|q_+(z,n)|=|\psi_n(z)|\quad \textrm{ when } n \textrm{ is odd.}
\end{align*}
Let $c=\sup_{n \in \Z_+} \|A(n,z)\|$. Then $|u_+(z,n)| \le c$ if $n$ is even and $|v_+(z,n)| \le c$ if $n$ is odd. By \cite[(3.2.23)]{S05}, $|p_+(z,n)| \le c^{-1}$ if $n$ is even and $|q_+(z,n)| \le c^{-1}$ if $n$ is odd. Thus,
\begin{align*}
&c^{-2} \le \frac{\|u_{+}(z,n)\|_{x(r)}}{\|p_{+}(z,n)\|_{x(r)}} \le c^2, \textrm{ when } n \textrm{ is even,}\\
&c^{-2} \le \frac{\|v_{+}(z,n)\|_{x(r)}}{\|q_{+}(z,n)\|_{x(r)}} \le c^2, \textrm{ when } n \textrm{ is odd.}
\end{align*}
It follows that $\mathcal{B}_+ \subseteq \mathcal{A}_+$.

The inclusion $\mathcal{B}_- \subseteq \mathcal{A}_-$ is proved in a similar way.
In section 4, we have proved $\CC_{-}$ conjugates to $\tilde{\CC}_{+}$ with $-\overline{\alpha}_{-(n+2)} = \tilde{\alpha}_{n}$ for $n\in \N$. Hence, we can rewrite the representation of $A(n,z)$ for $n \le -1$ as
$$
S(\tilde{\alpha}_{-n},z) \times S(\tilde{\alpha}_{-n+1},z) \times \cdots \times S(\tilde{\alpha}_{0},z).
$$
From the above statements, $\tilde{\CC}_{+}$ has no solution that is subordinate at $+\infty$. Due to the conjugation, \eqref{eq1} has no solution that is subordinate at $-\infty$. It follows that $\mathcal{B}_- \subseteq \mathcal{A}_-$.

The fact that the restriction of $\Lambda$ to each of $\mathcal{B}_{\pm}$ is purely absolutely continuous now follows from part (c) of Theorem~\ref{t.main1rev}.
\end{proof}

\begin{proof}[Proof of Corollary~\ref{c.ddvcbddc}]
If $z \in \mathcal{R}$, then it follows from the definition of $\mathcal{R}$ that there exist $B_z : \Omega \to \mathrm{SU}(1,1)$ bounded and $A_z^{(0)} \in \mathrm{SU}(1,1)$ elliptic such that for every $\omega \in \Omega$, we have $A_z(\omega) = B_z(T \omega) A_z^{(0)} B_z(\omega)^{-1}$. This in turn shows that for $n \ge 2$, we have
\begin{align*}
& A_z(T^{n-1} \omega) \times \cdots \times A_z(T \omega) A_z(\omega) \\
& = B_z(T^n \omega) A_z^{(0)} B_z(T^{n-1} \omega)^{-1} \times \cdots \times B_z(T^2 \omega) A_z^{(0)} B_z(T\omega)^{-1} B_z(T \omega) A_z^{(0)} B_z(\omega)^{-1} \\
& = B_z(T^n \omega) \left(A_z^{(0)}\right)^n B_z(\omega)^{-1}.
\end{align*}
Since the matrix on the left-hand side has the same norm as $A(n,z;\omega)$ (it is obtained from that matrix by multiplication with the unimodular number $z^{-n/2}$) and the right-hand side remains bounded as $n \to \infty$ since $A_z^{(0)}$ is elliptic and $B_z : \Omega \to \mathrm{SU}(1,1)$ is bounded, we find that
\begin{equation}\label{e.c23proofincl}
\mathcal{R} \subseteq \mathcal{B}_+ (\omega).
\end{equation}
We note in passing that a similar analysis applies on the left half line and yields $\mathcal{R} \subseteq \mathcal{B}_- (\omega)$.

Here we denote the matrices $A(n,z)$ and sets $\mathcal{B}_\pm$ introduced earlier for a fixed extended CMV matrix $\mathcal{E}$ by $A(n,z;\omega)$ and $\mathcal{B}_\pm (\omega)$ if they are associated with the dynamically defined extended CMV matrix $\mathcal{E}(\omega)$.

The assertion of Corollary~\ref{c.ddvcbddc} now follows from \eqref{e.c23proofincl}, which as discussed above holds for every $\omega \in \Omega$, and Corollary~\ref{c.bddsol}.
\end{proof}

\end{document}